\documentclass[11pt]{article}
\usepackage{amssymb,epsfig}
\usepackage{amsmath}
\usepackage{latexsym}
\usepackage{booktabs}
\usepackage{multirow,array}
\usepackage{textcomp}
\usepackage{eufrak}
\usepackage{amsthm}
\usepackage{hyperref}
\usepackage{inputenc}

\title{\bf Interpolation and approximation of piecewise smooth functions with corner discontinuities on sigma quasi-uniform grids.}

\author{J.A. Padilla \thanks{ Departamento de Matem\'atica Aplicada y Estad\'{\i}stica.
  Universidad Polit\'ecnica de Cartagena (Spain).
e-mail:{\tt joanpaso.ingnaval@gmail.com}
}
 \and J.C.
Trillo\thanks{ Departamento de Matem\'atica Aplicada y
Estad\'{\i}stica.
   Universidad  Polit\'ecnica de Cartagena (Spain).
e-mail:{\tt jc.trillo@upct.es.}
}
}

\DeclareRobustCommand{\perthousand}{%
  \ifmmode
    \text{\textperthousand}%
  \else
    \textperthousand
  \fi}

\begin{document}

\maketitle

\newtheorem{proposition}{Proposition}
\newtheorem{lemma}{Lemma}
\newtheorem{definition}{Definition}
\newtheorem{theorem}{Theorem}
\newtheorem{corollary}{Corollary}
\newtheorem{remark}{Remark}
\newtheorem{algorithm}{Algorithm}

\begin{abstract}
This paper provides approximation orders for a class of nonlinear interpolation procedures
for univariate data sampled over $\sigma$ quasi-uniform grids.
The considered interpolation is built using both essentially nonoscillatory
(ENO) and subcell resolution (SR) reconstruction techniques. The main target of these nonlinear techniques is to reduce
the approximation error for functions with isolated corner singularities and in turn this fact makes them useful
for applications to other fields, such as shock capturing computations or image processing. We start proving the approximation capabilities
of an algorithm to detect the
presence of isolated singularities, and then we address the approximation order attained by the mentioned interpolation procedure.
For certain nonuniform grids with a maximum spacing between nodes $h$  below a critical value $h_c$, the optimal approximation order is recovered, as it happens for uniformly smooth
functions \cite{ACDD}.

\end{abstract}

{\bf Key Words.} Interpolation, approximation, detection, discontinuities

\vspace{10pt} {\bf AMS(MOS) subject classifications.} 41A05,
41A10, 65D05.


\section{Introduction}

Nonlinear techniques emerge as a way to avoid undesirable effects of classical linear methods to tackle different kind of problems. When interpolating data coming from a smooth function affected by
a jump discontinuity, the results using linear methods show a bad behavior around the discontinuity, typically giving rise to diffusion of the jump and spurious oscillations known as Gibbs effects.
These problems can be avoided by treating the data differently depending on their characteristics, that is, in a data dependent way, and moreover implementing an adapted nonlinear approximation in these cases.
Some nonlinear methods such as ENO method \cite{HAR}, WENO method and subcell resolution \cite{HAR,ADO,R4}, PPH method \cite{ADLT} have emerged in a variety of applications with reasonably good results. Among
these applications we can mention signal processing, generation of curves and surfaces, subdivision and multiresolution schemes, numerical integration, and the numerical solution of certain differential equations.

It is quite common to work with uniform grids as a default case, and preferably because of simplicity and computational economy. However, some applications require dealing with nonuniform grids. Therefore,
some methods need to be adapted and the theory extended to guarantee the properties of the methods in this new scenario. The theoretical results are quite complicated to get with general nonuniform grids, if possible at all. Therefore, some extra restrictions are added to the nonuniform grids that make it possible to develop theoretical results at the same type that keep the properties of the grids that appear in practice. One extended definition of a particular nonuniform grid which cover almost all practical cases is the so called $\sigma$ quasi-uniform grid, i.e., a grid for which there exists a minimum grid spacing $h_{min},$ a maximum
grid spacing $h_{max}$ and $\frac{h_{max}}{h_{min}}\leq \sigma.$  Within this setting one can extend the theoretical results in a similar way as it is done for uniform grids, requiring more tedious computations, but
keeping the same essence as in the uniform case. For example, in \cite{OT3} the authors extend the work in \cite{ADLT} to $\sigma$ quasi-uniform grids in the mentioned way.

Our aim with this article is to generalize to $\sigma$ quasi-uniform grids the results given in \cite{ACDD} about a singularity detection mechanism and also the results about a subcell resolution type reconstruction operator defined in the same article. The singularity detection mechanism is interested on its own, since it can be applied in combination with many other reconstruction operators, and this is a reason to
make a detailed study of it.

The paper is organized as follows. In section \ref{sec2}, we extend an example in \cite{ACDD} to $\sigma$ quasi-uniform grids. This example shows that
one cannot expect more than second order accuracy when a corner singularity occurs. In section \ref{sec3}, the corresponding extension of both a specific corner
detection mechanism and an ENO-SR type interpolation process is carried out.
In section \ref{sec4}, we show following the same track as in \cite{ACDD} that detection always occurs for $h < h_c$ and that the position of the
singularity is accurately estimated. The results of section \ref{sec4} are used in section \ref{sec5}
to prove that the presented ENO-SR type interpolation process has accuracy of order $O(h^m)$
for $h$ smaller than $Kh_c$, where $0 < K < 1$ is a fixed constant, and that it is second
order accurate for all $h > 0,$ which is the best that we can hope for according to
the example of section 2, i.e. it happens the same behavior as with uniform grids.
In section \ref{sec6} we present some numerical examples to reinforce the previous theoretical results. Finally, in section \ref{sec7} we give some conclusions.


\section{An ENO-SR type interpolation and a clarifying example} \label{sec2}

We start exposing a breve summary of ENO and Subcell Resolution SR interpolation methods (see \cite{HAR,ACDD}) over nonuniform grids.
Let $X=(x_i)_{i \in \mathbb{Z}}$ be a nonuniform grid and let $f_i$ represent the pointvalues of a continuous function $f(x)$ with a corner discontinuity at $\mu \in [x_j,x_{j+1}]$ for a certain value of $j.$
Each interval $[x_i,x_{i+1}]$ is associated with a stencil $S_i$ consisting of $m$ points, $S_i=\{x_{i-m_1+1},\ldots,x_{i+m_2}\}$ with $m_1,m_2$ two integers satisfying $m_1+m_2=m, m_1>0, m_2>0.$ An interpolation procedure is built using this stencil just
by using Lagrange interpolation, that is,
\begin{equation} \label{s2_1}
I_Xf(x)=p_{i,m}(x), x \in [x_i,x_{i+1}],
\end{equation}
where $p_{i,m}(x)$ stands for the Lagrange interpolation for the stencil $S_i.$

Before addressing questions dealing with orders of approximation, we introduce an important definition.

\begin{definition}
A nonuniform grid $X=(x_i)_{i\in \mathbb{Z}}$ is said to be a $\sigma$ quasi-uniform grid if there exist $h_{min}=\min\limits_{i\in \mathbb{Z}} h_i,$
$h_{max}=\max\limits_{i\in \mathbb{Z}} h_i,$ and a finite constant $\sigma$ such that $\frac{h_{max}}{h_{min}}\leq \sigma,$ where $h_i:=x_{i}-x_{i-1}.$
\end{definition}

If the stencil $S_i$ does not touch the interval where the singularity lies, then it is well known that the order of approximation
is $O(h^m),$ with $h=h_{max}.$ If the stencil touches the corner singularity, then the approximation order is drastically lowered to $O(h),$
and this happens in $m-1$ intervals around the discontinuity. The number of affected intervals can be reduced to only one interval by using
ENO strategies. These strategies are based on a stencil selection mechanism which uses divided differences as indicators of the smoothness of a function.
The smaller the absolute value of the divided difference, the smoother the function in that area. To detect corner singularities it is enough to use
second order divided differences,
\begin{eqnarray} \label{difdiv}
D_if&=&\frac{1}{h_{i+1}(h_{i+1}+h_{i+2})}f_{i}-\frac{1}{h_{i+1}h_{i+2}}f_{i+1}+\frac{1}{h_{i+2}(h_{i+1}+h_{i+2})}f_{i+2}.\\ \notag
\end{eqnarray}
At an interval $I_i=[x_i,x_{i+1}],$  keeping the stencil with $m$ points that contains $I_i$ and gives the smallest divided difference leads us to the
definition of the ENO interpolation polynomial $\tilde{p}_{i,m}(x)$ as the Lagrange interpolation based on that data-dependent stencil.
However, ENO strategies continue to perform inadequately at one interval. In order to correct this situation, the SR procedure emerges. It is based on
a detection mechanism that allows to approximate the corner $\mu$ of the underlying function by $\psi$ with precision $O(h^m),$ and then define the interpolation as,
\begin{equation}\label{SR}
I_Xf(x)=\left\{ \begin{array}{cc}
\tilde{p}_{i,m}(x) & if \ x \in [x_i,x_{i+1}], i\neq j,\\
\tilde{p}_{j,m}^-(x) & if \ x \in [x_j,\psi],\\
\tilde{p}_{j,m}^+(x) & if \ x \in [\psi,x_{j+1}].\\
\end{array}\right.
\end{equation}
Notice that $\tilde{p}_{j,m}^-(\psi)$ must be equal to $\tilde{p}_{j,m}^+(\psi)$ for this interpolation to recover an appropriate approximation of the underlying function.

We aim at proving approximation results for this interpolation over $\sigma$ quasi-uniform grids, generalizing the work in \cite{ACDD} to nonuniform grids.
Let us consider functions $f \in C^m(\mathbb{R}\backslash \mu)$ with $m$ derivatives uniformly bounded on $\mathbb{R}\backslash \mu,$  which implies that $f$  has finite jumps
$[f'],[f''],\ldots$ at the corner discontinuity. Let us call the set of such functions as $\mathbb{A}.$ \\
Following the same track as in \cite{ACDD}, we are going to prove that also for $\sigma$ quasi-uniform grids we have,
\begin{equation} \label{s2_2}
||f-I_Xf||_{L_{\infty}} \leq Ch^m \sup_{\mathbb{R}\backslash \mu} |f^{(m)}(x)|, \ \textrm{if} \ h\leq Kh_c,
\end{equation}
with $0\leq K<1$ a constant and $h_c$ a certain critical value for the grid spacing. In case, $h>h_c$ it is possible to prove,
\begin{equation} \label{s2_3}
||f-I_Xf||_{L_{\infty}} \leq Ch^2 \sup_{\mathbb{R}\backslash \mu} |f^{''}(x)|.
\end{equation}
In general, getting (\ref{s2_2}) for any $h$ is not possible as it is shown with the next example, which is an adapted example for nonuniform grids of the
example given in \cite{ACDD} for uniform grids.\\
Given a $\sigma$ quasi-uniform grid $X,$ let us define two functions $f_{+}$ and $f_{-}$ in the following way,
\begin{eqnarray}\label{s2_4}
f_{+}(x)&=&0, \ \textrm{if} \ x\leq x_j, \quad f_{+}(x)=(x-x_j)(x-x_{j+1}), \ \textrm{if} \ x > x_j,\\ \notag
f_{-}(x)&=&0, \ \textrm{if} \ x\leq x_{j+1}, \quad f_{-}(x)=(x-x_j)(x-x_{j+1}), \ \textrm{if} \ x > x_{j+1}.\\ \notag
\end{eqnarray}
Both functions coincide on the grid points $X$ and therefore $I_Xf_{+}=I_Xf_{-}.$ Moreover,$||f_{+}-f_{-}||_{L_{\infty}}=\frac{h_{j+1}^2}{4},$ since
\begin{equation}\label{s2_5}
||f_{+}-f_{-}||_{L_{\infty}}=\max\limits_{x \in [x_j,x_{j+1}]}\{|(x-x_j)(x-x_{j+1})|\},
\end{equation}
and defining $p_2(x)=x^2-(x_j+x_{j+1})x+x_jx_{j+1},$ we see that it attains a maximum value at $x_{j+\frac{1}{2}}=\frac{x_j+x_{j+1}}{2},$ point where $p_2'(x_{j+\frac{1}{2}})=0$ and
$p_2(x_{j+\frac{1}{2}})=\frac{h_{j+1}^2}{4}.$ \\
Therefore, either $||f_{+}-I_Xf_{+}||_{L_{\infty}}\geq \frac{h_{j+1}^2}{8}$ or $||f_{-}-I_Xf_{-}||_{L_{\infty}}\geq \frac{h_{j+1}^2}{8},$ because otherwise we would have,
\begin{equation} \label{s2_6}
||f_{+}-f_{-}||_{L_{\infty}}\leq ||f_{+}-I_Xf_{+}||_{L_{\infty}}+||f_{-}-I_Xf_{-}||_{L_{\infty}} < \frac{h_{j+1}^2}{8}+ \frac{h_{j+1}^2}{8}=\frac{h_{j+1}^2}{4},
\end{equation}
what is a contradiction with the fact that $||f_{+}-f_{-}||_{L_{\infty}}=\frac{h_{j+1}^2}{4}.$
Thus, we have that at least one of the following inequalities is true,
\begin{eqnarray}
||f_{+}-I_Xf_{+}||_{L_{\infty}}&\geq& \frac{h_{j+1}^2}{8}\geq (\frac{h_{min}}{h_{max}})^2\frac{h_{max}^2}{8}\geq \frac{h^2}{8 \sigma^2} \geq \frac{h^2}{16 \sigma^2} \sup_{\mathbb{R}\backslash \mu} |f_{+}^{''}(x)|,\\ \notag
||f_{-}-I_Xf_{-}||_{L_{\infty}}&\geq& \frac{h_{j+1}^2}{8}\geq (\frac{h_{min}}{h_{max}})^2\frac{h_{max}^2}{8}\geq \frac{h^2}{8 \sigma^2} \geq \frac{h^2}{16 \sigma^2} \sup_{\mathbb{R}\backslash \mu} |f_{-}^{''}(x)|.\\ \notag
\end{eqnarray}
Taking also into account that $\sup_{\mathbb{R}\backslash \mu} |f_{+}^{(m)}(x)|=0$ and $\sup_{\mathbb{R}\backslash \mu} |f_{-}^{(m)}(x)|=0$ for $m>2,$ we observe that a bound of the type (\ref{s2_2}) is not always possible.
This means, that over $\sigma$ quasi-uniform grids, it can not be ensured that the interpolation given by (\ref{SR}) attains more than second order of accuracy for piecewise smooth functions in the previously
defined set $\mathbb{A}$ and for all $h>0.$

\section{A modified ENO-SR detection and interpolation mechanism over $\sigma$ quasi-uniform grids.} \label{sec3}
Both the detection and interpolation mechanisms in this section correspond with an extension to $\sigma$ quasi-uniform grids of the work presented in \cite{ACDD}.
We start adapting the corner detection mechanism. Given $m$ the desired order of approximation, the algorithm marks certain intervals of type $B$ meaning that they potentially
contain a corner singularity. This procedure of labeling the suspicious intervals is defined in two steps as follows,
\begin{itemize}
\item[1.] If
\begin{equation} \label{cond1}
|D_{i-1}f| > |D_{i-1 \pm n}f|, \ n=1,\ldots,m,
\end{equation}
then the intervals $I_{i-1}=[x_{i-1},x_i]$ and $I_{i}=[x_i,x_{i+1}]$ are labeled as $B.$
\item[2.] If
\begin{equation} \label{cond2}
|D_if| > |D_{i+n}|, n=1,\ldots,m-1,
\end{equation}
and
\begin{equation*} \label{cond3}
|D_{i-1}f| > |D_{i-1-n}|, n=1,\ldots,m-1,
\end{equation*}
then the interval $I_i=[x_i,x_{i+1}]$ is labeled as $B.$
\end{itemize}
The rest of intervals are marked as $G,$ what means that they are not suspected of containing a corner singularity.\\
Notice that the divided differences in (\ref{cond1}), (\ref{cond2}), and (\ref{cond3}) are computed by using the expressions for nonuniform grids given
in (\ref{difdiv}). \\

This detection mechanism will be proved to work well for the maximum grid spacing $h:=h_{max}$ associated with the $\sigma$ quasi-uniform grid small enough.
It will detect the interval containing the singularity, and it will ensure that the intervals marked as $G$ do not contain any corner singularity.
It could marked as $B$ some intervals which are $G,$ but this will be solve later on.

Let us introduce a needed parallel lemma to the one proved in \cite{ACDD}. The proof is quite similar and we will not include it. See \cite{ACDD} for more details.

\begin{lemma}
The groups of adjacent $B$ intervals are at most of size $2.$ They are separated by groups of adjacent $G$ intervals which are at least of size $m-1.$
\end{lemma}

Following the same track as in \cite{ACDD}, we define a interpolation procedure in the following way.
\begin{itemize}
\item[1.] If $I_i=[x_{i},x_{i+1}]$ is of type $G,$ then the interpolation $I_Xf(x)$ on $I_i$ amounts to the polynomial $p_{i,m}(x)$ of degree $m-1$ obtained using ENO strategies to select a stencil of $m$
points containing $x_i$ and $x_{i+1}$ and staying within the smoothest part of the function.
\item[2.] If $I_i=[x_{i},x_{i+1}]$ is an isolated $B$-type interval, then we build the interpolatory polynomial $\tilde{p}_{i}^{-}(x)$ based on the stencil $\{x_{i-m+1},\ldots,x_i\}$ and the polynomial $\tilde{p}_{i}^+(x)$ based on the
stencil $\{x_{i+1},\ldots,x_{i+m}\},$ and we use them to predict the location of the corner. If both polynomials intersect at a unique point $\psi,$ then,
\begin{equation*}\label{SR1}
I_Xf(x)=\left\{ \begin{array}{cc}
\tilde{p}_{i,m}^-(x) & if \ x \in [x_i,\psi],\\
\tilde{p}_{i,m}^+(x) & if \ x \in [\psi,x_{i+1}].\\
\end{array}\right.
\end{equation*}
If the polynomials do not intersect, then we redefine the interval as type $G,$ and we return to the first point.
\item[3.] If both $I_i=[x_{i},x_{i+1}]$ and $I_{i+1}=[x_{i+1},x_{i+2}]$ are labeled as $B$, then we treat $I=I_i\cup I_{i+1}$ as a unique interval, and we proceed in the same way as in the previous point. We build the interpolatory polynomial $p_{i}^{-}(x)$ based on the stencil $\{x_{i-m+1},\ldots,x_i\}$ and the polynomial $p_{i+1}^+(x)$ based on the
stencil $\{x_{i+2},\ldots,x_{i+m+1}\},$ and we use them to predict the location of the corner. If both polynomials intersect at a unique point $\psi,$ then,
\begin{equation*}\label{SR1}
I_Xf(x)=\left\{ \begin{array}{cc}
\tilde{p}_{i,m}^-(x) & if \ x \in [x_i,\psi],\\
\tilde{p}_{i+1,m}^+(x) & if \ x \in [\psi,x_{i+2}].\\
\end{array}\right.
\end{equation*}
If the polynomials do not intersect, then we redefine both intervals $I_i$ and $I_{i+1}$ as $G,$ and we apply the first point.
Notice that in case of getting an intersection point $\psi$ in the interval $I,$ then the reconstruction $I_Xf(x)$ does not interpolate the original function at $x_{j+1}.$
\end{itemize}

\section{Properties of the detection mechanism} \label{sec4}

In this section the main features of the detection mechanism are studied. The content is given by means of two lemmas. In the first one, it is ensured that the corner discontinuity will
be detected under a critical value of the diameter of the grid $h_{max}.$ This study is parallel to the one carried out in \cite{ACDD} for uniform grids.
\begin{lemma} \label{lem1proDe}
Let $f(x)$ be a continuous function on $\mathbb{R}\backslash \mu,$ with a bounded second derivative on $\mathbb{R}\backslash \mu,$ with $\mu$ a discontinuity in the first derivative of
the function. Let us define the critical value for the grid,
\begin{equation} \label{lemma2_0}
h_c:=\frac{|[f']|}{4 \sup_{\mathbb{R}\backslash \mu} |f^{''}(x)|},
\end{equation}
where $[f']$ denotes the jump of the first derivative at $\mu.$ Then, for $h_{max}<h_c,$ the interval $I_j=[x_j,x_{j+1}]$ which contains $\mu$ is marked as type $B.$ If $\mu$ is close to one
of the extremes of the interval $I_j$ in such a way that,
\begin{equation} \label{lemma2}
\frac{x_{j+1}-\mu}{h_j+h_{j+1}}-\frac{\mu-x_{j}}{h_{j+1}+h_{j+2}}>1/4 \qquad \textrm{or} \qquad \frac{\mu-x_{j}}{h_{j+1}+h_{j+2}}-\frac{x_{j+1}-\mu}{h_j+h_{j+1}}>1/4
\end{equation}
then the corresponding adjacent interval is also marked as type $B.$
\end{lemma}

\begin{proof}
Without loss of generality, we can admit that $\mu$ is located on the first half of the interval $I_j,$ i.e., $x_j\leq \mu \leq x_{j+\frac{1}{2}}=\frac{x_j+x_{j+1}}{2}.$ For $i\leq j-2$ and $i\geq j+1$ we have,
\begin{equation} \label{lemma2_1}
|D_if|=|f[x_i,x_{i+1},x_{i+2}]|=\frac{f''(\theta_i)}{2}|\leq \frac{1}{2}\sup_{\mathbb{R}\backslash \mu} |f^{''}(x)|,
\end{equation}
with $\theta_i$ an intermediate point between $x_i$ and $x_{i+2}.$\\
For $i=j-2$ and $i=j+1,$ the second order divided differences can be approximated by decomposing $f=f_1+f_2,$ with,
\begin{equation} \label{lemma2_2}
f_1(x)=[f'](x-\mu)_{+}=\left\{\begin{array}{cc}
[f'](x-\mu) & \textrm{if} \ x\geq \mu,\\
0 & \textrm{otherwise,}
\end{array}\right.
\qquad f_2(x)=f(x)-f_1(x).
\end{equation}
Notice that $f_2(x)$ is a $C^1$ function with a bounded second derivative on $\mathbb{R}\backslash \mu$ such that
\begin{equation}\label{lemma2_3}
\sup_{\mathbb{R}\backslash \mu} |f_2^{''}(x)|=\sup_{\mathbb{R}\backslash \mu} |f^{''}(x)|.
\end{equation}
\end{proof}
For all $i \in \mathbb{Z}$ we have,
\begin{equation}\label{lemma2_4}
|D_i f| \leq \frac{1}{2}\sup_{\mathbb{R}\backslash \mu} |f^{''}(x)|.
\end{equation}
We also have,
\begin{equation}\label{lemma2_5}
|D_{j-1}f_1|=|\frac{f_1(x_{j-1})}{h_{j}(h_j+h_{j+1})}-\frac{f_1(x_j)}{h_jh_{j+1}}+\frac{f_1(x_{j+1})}{h_{j+1}(h_j+h_{j+1})}|=|\frac{(x_{j+1}-\mu)[f']}{h_{j+1}(h_j+h_{j+1})}|.
\end{equation}
From (\ref{lemma2_4}) and (\ref{lemma2_5}), it follows that,
\begin{eqnarray}\label{lemma2_6}
|D_{j-1}f|&=&|D_{j-1}f_1 + D_{j-1}f_2|\geq |D_{j-1}f_1|-|D_{j-1}f_2|\geq  |\frac{(x_{j+1}-\mu)[f']}{h_{j+1}(h_j+h_{j+1})}| - \frac{1}{2}\sup_{\mathbb{R}\backslash \mu} |f^{''}(x)|\\ \notag
&\geq& |\frac{[f']}{2(h_j+h_{j+1})}| - \frac{1}{2}\sup_{\mathbb{R}\backslash \mu} |f^{''}(x)|.\\ \notag
\end{eqnarray}
Denoting $h=h_{max},$ from (\ref{lemma2_6}) we get,
\begin{equation} \label{lemma2_7}
|2h^2D_{j-1}f|\geq \frac{h}{2}|[f']|- h^2\sup_{\mathbb{R}\backslash \mu} |f^{''}(x)|.
\end{equation}

Thus, if $h=h_{max}<h_c$ with $h_c$ given in (\ref{lemma2_0}), then,
\begin{equation} \label{lemma2_8}
|2h^2D_{j-1}f|> h^2\sup_{\mathbb{R}\backslash \mu} |f^{''}(x)|.
\end{equation}

By the transitivity of the inequalities, using (\ref{lemma2_1}) and (\ref{lemma2_8}), we get that for $h<h_c,$
\begin{equation} \label{lemma2_9}
|D_{j-1}f|> |D_if|,
\end{equation}
for all  $i\leq j-2$ and $i\geq j+1.$\\

If $|D_jf|<|D_{j-1}f|,$ then both intervals $I_{j-1}$ and $I_j$ are labeled as type $B$ because of (\ref{cond1}). Otherwise, if
$|D_jf|\geq|D_{j-1}f|,$ then $I_j$ is marked as $B$ due to (\ref{cond2}).

Finally,

\begin{eqnarray} \label{lemma2_10}
|2h^2D_jf_1|&=&2h^2|\frac{f_1(x_{j})}{h_{j+1}(h_{j+1}+h_{j+2})}-\frac{f_1(x_{j+1})}{h_{j+1}h_{j+2}}+\frac{f_1(x_{j+2})}{h_{j+2}(h_{j+1}+h_{j+2})}|\\ \notag
&=&2h^2|-\frac{f_1(x_{j+1})}{h_{j+1}h_{j+2}}+\frac{f_1(x_{j+2})}{h_{j+2}(h_{j+1}+h_{j+2})}|\\ \notag
&=& \frac{2h^2}{h_{j+1}(h_{j+1}+h_{j+2})} (\mu-x_j)|[f']|.\\ \notag
\end{eqnarray}

Thus,
\begin{equation} \label{lemma2_11}
|2h^2D_jf|\leq |2h^2D_jf_1|+|2h^2D_jf_2|\leq \frac{2h^2}{h_{j+1}(h_{j+1}+h_{j+2})} (\mu-x_j)|[f']| + h^2\sup_{\mathbb{R}\backslash \mu} |f^{''}(x)|.
\end{equation}

Imposing the condition in (\ref{lemma2}), we ensure that,
\begin{equation} \label{lemma2_12}
\frac{2h^2}{h_{j+1}(h_{j+1}+h_{j+2})} (\mu-x_j)|[f']| + h^2\sup_{\mathbb{R}\backslash \mu} |f^{''}(x)| < 2h^2|\frac{(x_{j+1}-\mu)[f']}{h_{j+1}(h_j+h_{j+1})}| - h^2\sup_{\mathbb{R}\backslash \mu} |f^{''}(x)|,
\end{equation}
and therefore due to (\ref{lemma2_6}), (\ref{lemma2_11}) and (\ref{lemma2_12}) we get,
\begin{equation*}
|2h^2D_jf|<|2h^2D_{j-1}f|,
\end{equation*}
what means that for $h<h_c$ if the condition (\ref{lemma2}) is satisfied, then the intervals $I_{j-1}$ and $I_j$ are both labeled type $B.$

In the next lemma, and following again the same track as in \cite{ACDD}, we prove that for $h<h_c,$ being $h_c$ the critical scale given in (\ref{lemma2_0}),
the position of the corner singularity is accurately detected.

\begin{lemma} \label{lem3proDe}
Let $f(x)$ be a continuous function on $\mathbb{R}\backslash \mu,$ with uniformly bounded $mth$-derivative on $\mathbb{R}\backslash \mu,$ with $\mu$ a discontinuity in the first derivative of
the function. Let $X$ be a $\sigma$ quasi-uniform grid with $\sigma < \frac{3}{2}.$ Let us define the critical value for the grid,
\begin{equation} \label{lemma2_0}
h_c:=\frac{|[f']|}{4 \sup_{\mathbb{R}\backslash \mu} |f^{''}(x)|},
\end{equation}
where $[f']$ denotes the jump of the first derivative at $\mu.$ Then, there exist constants $C,k, \ C>0, \ 0 < k < 1$ such that, for $h_{max}<k h_c,$ the following items hold,
\begin{itemize}
\item[1.] The singularity $\mu$ is inside a $B$-type interval $I_j$ (Case 1) or in a $B$-pair of intervals $I_{j-1}, I_j.$
\item[2.] The two polynomials $p_j^-,p_{j}^+$ (Case 1) or $p_{j-1}^-,p_{j}^+$ (Case 2), which appear in the definition of $I_Xf,$ have a unique intersection point
$y$ in the interval $I_j$ (Case 1) or in $I_{j-1}\cup I_j$ (Case 2).
\item[3.] The distance between the real singularity $\mu$ and its estimation $y$ is bounded by,
\begin{equation} \label{lemma3_1}
|\mu -y|\leq C \frac{h^m \sup_{\mathbb{R}\backslash \mu} |f^{(m)}(x)|}{|[f']|},
\end{equation}
where $C>0$ is a given constant which depends on the grid.
\end{itemize}
\end{lemma}

\begin{proof}
Taking into account that $k<1$ the first point has been already proven in Lemma \ref{lem1proDe}.
Without loss of generality, we consider that $x_j \leq \mu \leq \frac{x_j+x_{j+1}}{2}=x_M.$ Due to Lemma \ref{lem1proDe} we know that $I_j$ is of type $B$ for $h_{max}<h_c.$
Let us denote $I=[a,b]$ the interval where the subcell resolution type algorithm takes place, either $I_j=[x_{j},x_{j+1}]$ in Case 1 or $I_{j-1}\cup I_j=[x_{j-1},x_{j+1}]$ in Case 2.\\
By Lemma \ref{lem1proDe}, we get that if $\frac{x_{j+1}-\mu}{h_j+h_{j+1}}-\frac{\mu-x_{j}}{h_{j+1}+h_{j+2}}>1/4,$ then $I=I_{j-1}\cup I_j,$ and for all cases we have,
\begin{equation}\label{lemma3_2}
\min\{|\mu-a|,|\mu-b|\} >\frac{3-2\sigma}{4\sigma^2}h.
\end{equation}

Let us also denote as $p^-,$ and $p^+$ the polynomials which appear in the subcell resolution type algorithm of $I.$ We decompose $f$ as,
\begin{equation}\label{lemma3_3}
f=f^- \cdot \chi_{(-\infty,a]} + f^+ \cdot \chi_{[a,\infty)},
\end{equation}
where $f^-$ and $f^+$ are defined as extensions of $f$ by using its left and right Taylor expansions of second order at $\mu$ respectively. Therefore,
$f^-$ and $f^+$ are functions which are globally $C^2$ in $\mathbb{R}$ and they satisfy,
\begin{equation}\label{lemma3_4}
\sup_{\mathbb{R}} |(f^{\pm})^{''}(x)|\leq \sup_{\mathbb{R}\backslash \mu} |f^{''}(x)|.
\end{equation}
Notice that $p^-$ and $p^+$ can be viewed as Lagrange interpolation polynomials, and then for all $t \in I,$ there exists a constant $D$ such that,
\begin{equation}\label{lemma3_5}
|f^{\pm}(t)-p^{\pm}(t)|\leq D h^2
\sup_{\mathbb{R}} |(f^{\pm})^{''}(x)|\leq D h^2 \sup_{\mathbb{R}\backslash \mu} |f^{''}(x)|,
\end{equation}
and therefore,
\begin{equation}\label{lemma3_6}
|(f^{\pm})'(t)-(p^{\pm})'(t)|\leq D h
\sup_{\mathbb{R}} |(f^{\pm})^{''}(x)|\leq D h \sup_{\mathbb{R}\backslash \mu} |f^{''}(x)|.
\end{equation}
Taking into account that $|t-\mu|\leq 2h,$ for $t \in I,$ we also get using the Lagrange mean value theorem,
\begin{equation}\label{lemma3_7}
|(f^{\pm})'(t)-(f^{\pm})'(\mu)|\leq 2h \sup_{\mathbb{R}\backslash \mu} |f^{''}(x)|.
\end{equation}
Hence from (\ref{lemma3_6}) and (\ref{lemma3_7}) we obtain,
\begin{equation}\label{lemma3_8}
|(f^{\pm})'(\mu)-(p^{\pm})'(t)|\leq (D+2)h \sup_{\mathbb{R}\backslash \mu} |f^{''}(x)|.
\end{equation}

Now, using the triangular inequality we get,
\begin{equation}\label{lemma3_9}
|(f^{+})'(\mu)-(f^{-})'(\mu)|=|[f']| \leq |(f^{+})'(\mu)-(p^{+})'(t)|+|(p^{+})'(t)-(p^{-})'(t)|+|(p^{-})'(t)-(f^{-})'(\mu)|,
\end{equation}
and then,
\begin{eqnarray}\label{lemma3_9}
|(p^{+})'(t)-(p^{-})'(t)|&\geq&|[f']| -|(f^{+})'(\mu)-(p^{+})'(t)|-|(p^{-})'(t)-(f^{-})'(\mu)|\\ \notag
&\geq& |[f']| - 2(D+2)h \sup_{\mathbb{R}\backslash \mu} |f^{''}(x)|.\\ \notag
\end{eqnarray}

For $h<\frac{2}{D+2}h_c,$ from (\ref{lemma3_9}) we get that the function $p^+(x)-p^-(x)$ is strictly increasing or strictly decreasing and
in turn that it has at most one root in the interval $I.$ It remains to see that under the given conditions, there exists one root $y \in I.$
Let us suppose that $[f']>0$ (the other case $[f']<0$ is analogous). Following the same track as in \cite{ACDD} one can easily prove using Taylor expansions
as well for nonuniform grids that,
\begin{equation}\label{lemma3_10}
f^{+}(a)-f^{-}(a)\leq -(\mu-a)[f']+(\mu-a)^2 \sup_{\mathbb{R}\backslash \mu} |f^{''}(x)|,
\end{equation}
\begin{equation}\label{lemma3_11}
f^{+}(b)-f^{-}(b)\geq (b-\mu)[f']-(b-\mu)^2 \sup_{\mathbb{R}\backslash \mu} |f^{''}(x)|,
\end{equation}
and using (\ref{lemma3_5}),
\begin{equation}\label{lemma3_12}
p^{+}(a)-p^{-}(a)\leq -(\mu-a)[f']+((\mu-a)^2+2Dh^2) \sup_{\mathbb{R}\backslash \mu} |f^{''}(x)|,
\end{equation}
\begin{equation}\label{lemma3_13}
p^{+}(b)-p^{-}(b)\geq (b-\mu)[f']-((b-\mu)^2+2Dh^2) \sup_{\mathbb{R}\backslash \mu} |f^{''}(x)|.
\end{equation}

Now, using (\ref{lemma3_2}) we obtain,
\begin{equation}\label{lemma3_14}
p^{+}(a)-p^{-}(a)\leq -\frac{3-2\sigma}{4\sigma^2}h[f']+(4+2D)h^2 \sup_{\mathbb{R}\backslash \mu} |f^{''}(x)|,
\end{equation}
\begin{equation}\label{lemma3_15}
p^{+}(b)-p^{-}(b)\geq \frac{3-2\sigma}{4\sigma^2}h[f']-(4+2D)h^2 \sup_{\mathbb{R}\backslash \mu} |f^{''}(x)|.
\end{equation}

For $h<kh_c,$ with $k=\frac{2}{2+D}\frac{3-2\sigma}{4\sigma^2},$ from (\ref{lemma3_14}) and (\ref{lemma3_15}) we get that $p^{+}(a)-p^{-}(a)<0,$ and $p^{+}(b)-p^{-}(b)>0,$ and
therefore there exists a root $y \in I$ of $p^+(x)-p^-(x),$ and the two first points of the lemma are already proven.

To prove the third point we observe that $f^-$ and $f^+$ could be also defined as extensions of $f$ by using its left and right Taylor expansions of $mth$ order at $\mu$ respectively. Then,
$f^-$ and $f^+$  become functions which are globally $C^m$ in $\mathbb{R}$ and they satisfy,
\begin{equation}\label{lemma3_16}
\sup_{\mathbb{R}} |(f^{\pm})^{(m)}(x)|\leq \sup_{\mathbb{R}\backslash \mu} |f^{(m)}(x)|.
\end{equation}
By using known results about Lagrange interpolation it follows that there exists a constant $\tilde{D}>0$ such that $\forall t \in I,$
\begin{equation}\label{lemma3_17}
|f^{\pm}(t)-p^{\pm}(t)|\leq \tilde{D} h^m \sup_{\mathbb{R}\backslash \mu} |f^{(m)}(x)|.
\end{equation}
Thus, defining $g=f^{+}-f^{-}$ and $q=p^{+}-p^{-}$ we get,
\begin{equation}\label{lemma3_17}
|g(t)-q(t)|\leq 2\tilde{D} h^m \sup_{\mathbb{R}\backslash \mu} |f^{(m)}(x)|.
\end{equation}
We also observe that for $h<Kh_c,$ with $K:=\min\{\frac{1}{4D+8},\frac{2}{D+2}\frac{3-2\sigma}{4\sigma^2}\}$ we have,
\begin{eqnarray}\label{lemma3_18}
|q'(t)|=|(p^{+})'(t)-(p^{-})'(t)|&\geq&|[f']| - 2(D+2)h \sup_{\mathbb{R}\backslash \mu} |f^{''}(x)|\\ \notag
&\geq& |[f']|-\frac{1}{8}|[f']|=\frac{7}{8}|[f']|. \notag
\end{eqnarray}
Thus,
\begin{equation}\label{lemma3_19}
2\tilde{D} h^m \sup_{\mathbb{R}\backslash \mu} |f^{(m)}(x)|\geq |q(\mu)|=|q(y)-q(\mu)|=|q'(\theta)||y-\mu|\geq \frac{7}{8}|[f']| |y-\mu|,
\end{equation}
and hence,
\begin{equation} \label{lemma3_20}
|\mu -y|\leq C \frac{h^m \sup_{\mathbb{R}\backslash \mu} |f^{(m)}(x)|}{|[f']|},
\end{equation}
with $C=\frac{16}{7}\tilde{D}.$

\end{proof}

The results in this section are used in the next one to prove a theorem about the approximation capabilities of the propose ENO-SR type interpolation.

\section{Approximation properties of the proposed ENO-SR interpolation.} \label{sec5}
In this section we prove an approximation result for the ENO-SR interpolation defined in
section \ref{sec2}.
\begin{theorem} \label{teo_approF}
	Let $f(x)$ be a continuous function on $\mathbb{R}\backslash \mu,$ with uniformly bounded $mth$-derivative on $\mathbb{R}\backslash \mu,$ with $\mu$ a discontinuity in the first derivative of
	the function. Let $X$ be a $\sigma$ quasi-uniform grid with $\sigma < \frac{3}{2}.$ Then,
	the nonlinear approximation $I_Xf$ satisfies,
	\begin{equation}\label{eq1_sec5}
		||f-I_Xf||_{\infty}\leq Ch^2  \sup_{\mathbb{R}\backslash \mu} |f^{''}(x)|,
	\end{equation}
for all $h=h_{max}>0,$ with $C$ independent on $f.$ Moreover, there exists
	a constant $k, \ 0 <k< 1$ such that, for $h<k h_c,$ with $h_c$ defined in (\ref{lemma2_0}), the following item holds,
	\begin{equation}\label{eq2_sec5}
		||f-I_Xf||_{\infty}\leq Ch^m \sup_{\mathbb{R}\backslash \mu} |f^{(m)}(x)|.
	\end{equation}
\end{theorem}
\begin{proof}
	It is clear by construction that for the value of $k$ given in previous lemmas, if the function is of class $C^m$ in the support of the local reconstruction, then the reconstruction satisfies,
	\begin{equation}\label{eq2b_sec5}
		|f(x)-I_Xf(x)|\leq Ch^m \sup_{\mathbb{R}\backslash \mu} |f^{(m)}(x)|,
	\end{equation}
	 due to classical Lagrange interpolation. This happens as much in a $G$ interval, as in a $B$ interval that is a false detection.\\
	Let us assume now that $x$ belongs to a pair of adjacent intervals of type $B,$ which contain the singularity $\mu.$ We assume, without loss of generality, that $x_j \leq \mu \leq x_j+\frac{h_{j+1}}{2},$ and let us use again the notation $I=[a,b], f^{\pm},p^{\pm}$ that was already used in Lemma \ref{lem3proDe}. Let us also assume that $\mu \leq y,$ being the case $\mu>y$ pretty similar. For $x \in [a,\mu],$ we have the estimate,
	\begin{equation}\label{eq3_sec5}
	|f(x)-I_Xf(x)|\leq |f^{-}(x)-p^{-}(x)|\leq Ch^m \sup_{\mathbb{R}\backslash \mu} |f^{(m)}(x)|,
    \end{equation}
	and for $x \in [y,b],$
	\begin{equation}\label{eq3b_sec5}
		|f(x)-I_Xf(x)|\leq |f^{+}(x)-p^{+}(x)|\leq Ch^m \sup_{\mathbb{R}\backslash \mu} |f^{(m)}(x)|.
	\end{equation}
	
	It remains to study the case $\mu < x < y.$ We can write,
	\begin{equation}\label{eq4_sec5}
	|f(x)-I_Xf(x)|= |f^{+}(x)-p^{-}(x)|\leq |f^{+}(x)-f^{-}(x)| + |f^{-}(x)-p^{-}(x)|.
    \end{equation}	
	The second term of the right hand side of (\ref{eq4_sec5}) is bounded by $C_1h^m\sup_{\mathbb{R}\backslash \mu} |f^{(m)}(x)| $ just by the classical theory of  Lagrange interpolation. For the first term,
	we consider Taylor expansions of second order and we get,
	\begin{eqnarray}\label{eq5_sec5}
		|f^{+}(x)-f^{-}(x)|&\leq& |[f'](y-\mu)|+ |y-\mu|^2 \sup_{\mathbb{R}\backslash \mu} |f^{''}(x)|,\\ \notag
        &=& |y-\mu|(|[f']|+|y-\mu|\sup_{\mathbb{R}\backslash \mu} |f^{''}(x)|).		
	\end{eqnarray}	
	By using Lemma \ref{lem3proDe}, we can see that there exists a constant $C_2$ such that $|[f']|+|y-\mu|\sup_{\mathbb{R}\backslash \mu} |f^{''}(x)|\leq C_2 |[f']|,$ and plugging this information into (\ref{eq5_sec5}) and applying again Lemma \ref{lem3proDe}, we get
	$|f^{+}(x)-f^{-}(x)|\leq D h^m \sup_{\mathbb{R}\backslash \mu}  |f^{(m)}(x)|, $ with $D$ constant.\\
	This finishes the proof for the case $h<k h_c.$ For $h\geq k h_c,$ the estimate,
		\begin{equation}\label{eq6_sec5}
		|f^{+}(x)-f^{-}(x)|\leq C_3 h^m \sup_{\mathbb{R}\backslash \mu}  |f^{(m)}(x)|,
	\end{equation}	
	is valid if $|x-\mu|\geq (m+1)h_{min},$ and therefore we also have,
	 \begin{equation}\label{eq7_sec5}
	 	|f^{+}(x)-f^{-}(x)|\leq C_4 h^2 \sup_{\mathbb{R}\backslash \mu}  |f^{''}(x)|.
	 \end{equation}	
	
	Let us now prove that the estimate in (\ref{eq7_sec5}) is also valid if $|x-\mu|< (m+1)h_{min}.$ This fact comes from decomposing $f=f_1+f_2$ in the same way as in
	 Lemma \ref{lem1proDe}. The interpolation errors for $f_1$ and $f_2$ are dominated by
	 $C_5h[f']$ and $C_6h^2\sup_{\mathbb{R}\backslash \mu}  |f^{''}(x)|$ respectively. And taking into account that $h\geq k h_c,$ the first bound transforms,
	 \begin{equation}\label{eq8_sec5}
	 	C_5h[f']\leq 4C_5hh_c\sup_{\mathbb{R}\backslash \mu}  |f^{''}(x)|\leq \frac{4C_5}{k}h^2 \sup_{\mathbb{R}\backslash \mu}  |f^{''}(x)|,
	 \end{equation}	
 what finishes this case and the proof.
\end{proof}

\begin{remark}
Lemma \ref{lem3proDe} is also true for $3-$local $\sigma$ quasi-uniform grids with $\sigma<\frac{3}{2}$ according to the following definition.
\begin{definition}
A nonuniform mesh $X=(x_i)_{i\in \mathbb{Z}}$ is said to be $n-$local  $\sigma$ quasi-uniform grid if for any $n$ consecutive grid spacings $\{h_i,\ldots,h_{i+n-1}\}$ there exist  a finite constant $\sigma$ such that $\frac{h_{max}}{h_{min}}\leq \sigma,$ where $h_{min}=\min\limits_{s=0,\ldots,n-1} h_{i+s},$
$h_{max}=\max\limits_{s=0,\ldots,n-1} h_{i+s}.$
\end{definition}
\end{remark}

\section{Numerical experiments} \label{sec6}

Let us consider the following functions depending on a parameter $d$
\begin{equation} \label{funcionf}
f_d(x):=
	\left\{\begin{array}{ll}
		(x-\frac{\pi}{8})^2+d(x-\frac{\pi}{8})+\cos{(\frac{\pi x}{2})}, & -1\leq x \leq \frac{\pi}{8},\\
		\cos{(\frac{\pi x}{2})}, & \frac{\pi}{8}<x\leq 1,\\
	\end{array} \right.
\end{equation}
where $d \ \in \{4,2,1,\frac{1}{2},\frac{1}{4},\frac{1}{8},\frac{1}{16},\frac{1}{32},\frac{1}{64},\frac{1}{128}\}.$ These functions present a discontinuity in the first derivative at
the point $\mu=\frac{\pi}{8}.$ The smaller $d,$ the weaker the discontinuity.

We carry out two numerical experiments, one addressing the approximation order locating the position of the corner discontinuity, and another one to measure the approximation errors and
the approximation orders when interpolating with the presented technique over $\sigma$ quasi-uniform grids.

For the first experiment, we consider a $\sigma$ quasi-uniform grid $X^0=(x_i^0)_{i=0}^{n}$ with $\sigma=2,$ consisting of $n+1=22$ non equally spaced points, $x_0^0=-1,$ $x_{21}^0=1.$
Then, we perform a process of subdivision of the grid, giving rise to successive grids $X^k$ obtained by computing $x_{2j}^k=x_{j}^{k-1},$ $x_{2j+1}^k=\frac{x_j^{k-1}+x_{j+1}^{k-1}}{2}.$ For each grid resolution, we run the corner detection mechanism with
$m=4,$ for the different values of $d.$ Since for the defined test function the critical spacing
$$h_c=\frac{|[f'_d]|}{4\sup_{x \in [-1,1]\backslash \frac{\pi}{8}}|f''_d(x)|}=\frac{d}{8},$$
according to the proven theoretical results, we expect to attain fourth order of accuracy for grids
where $h_{max}<h_c.$ For each scale we have computed the error in location of the corner discontinuity $e_k=|\mu - \psi|,$ where $\psi$ denotes the computed discontinuity approximation and $\mu$ the exact location.
After that, we compute a sequence of numerical approximation orders by calculating
$$p_k=\frac{log_2{e_{k-1}}}{log_2{e_{k}}}.$$
In Table \ref{tab1}, we can see the different values of $h_{max}$ at the considered grids.
In Table \ref{tab2}, we can observe that for larger values of $d$ the algorithm attains fourth order of approximation for all scales. When $d$ becomes smaller and smaller, the first scales do not give fourth order of accuracy, since the grid spacing is not small enough. However, we observe that even for slightly larger $h$ than $h_c$ the fourth order accuracy is targeted.

\begin{table}[!ht]
\begin{center}
\begin{tabular}{|c|c|c|c|c|c|c|c|}
\hline
$k$  & $0$ & $1$ & $2$ & $3$ & $4$ & $5$ & $6$ \\
\hline
$h_{max}$ & $0.1270$ & $0.0635$ & $0.0317$ & $0.0159$ & $0.0079$ & $0.0040$ & $0.0020$\\
\hline
\end{tabular}
\caption{\small\textit{Values of the maximum grid spacing $h_{max}$ for each grid $X^k,$ where $k$ denotes the scale.}} \label{tab1}
\end{center}
\end{table}

\begin{table}[!ht]
\begin{center}
\begin{tabular}{|c|c|c|c|c|c|}
\hline
Cases  & \multicolumn{5}{|c|}{Refinement grid level $k$}  \\
\hline \multirow{4}{*}{Parameter $d$}  &  $e_0$ & $e_1$ & $e_2$ & $e_3$ & $e_4$ \\
 \cline{2-6} & $p_0$ & $p_1$ & $p_2$ & $p_3$ & $p_4$ \\
 \cline{2-6} & $e_5$ & $e_6$ &  & &  \\
 \cline{2-6} & $p_5$ & $p_6$ &  & &  \\
\hline \multirow{3}{*}{$d=4$}  &  2.9630e-05 & 1.3626e-06 & 1.6274e-07 & 8.6763e-09 & 4.2243e-10  \\
 \cline{2-6} & - & 4.4427 & 3.0657 & 4.2293 &  4.3603 \\
\cline{2-6} &  1.2697e-11 & 8.0550e-13 &  &  &  \\
 \cline{2-6} & 5.0561 & 3.9785 &  &  & \\
\hline \multirow{3}{*}{$d=2$}  &  5.9371e-05 & 2.7258e-06 & 3.2548e-07 & 1.7353e-08 & 8.4487e-10  \\
 \cline{2-6} & - & 4.4450 & 3.0660 & 4.2293 &  4.3603 \\
\cline{2-6} &  2.5395e-11 & 1.6097e-12 &  &  &  \\
 \cline{2-6} & 5.0561 & 3.9797 &  &  & \\
\hline \multirow{3}{*}{$d=1$}  &  1.1918e-04 & 5.4543e-06 & 6.5100e-07 & 3.4706e-08 & 1.6897e-09  \\
 \cline{2-6} & - & 4.4496 & 3.0667 & 4.2294 &  4.3603 \\
\cline{2-6} &  5.0791e-11 & 3.2187e-12 &  &  &  \\
 \cline{2-6} & 5.0561 & 3.9800 &  &  & \\
\hline \multirow{3}{*}{$d=5e-01$}  &  2.4007e-04 & 1.0919e-05 & 1.3021e-06 & 6.9412e-08 & 3.3795e-09  \\
 \cline{2-6} & - & 4.4585 & 3.0679 & 4.2295 &  4.3603\\
\cline{2-6} &  1.0158e-10 & 6.4376e-12 &  &  &  \\
 \cline{2-6} & 5.0561 & 3.9800 &  &  & \\
\hline \multirow{3}{*}{$d=2.5e-01$}  &  4.8680e-04 & 2.1881e-05 & 2.6047e-06 & 1.3883e-07 & 6.7590e-09  \\
 \cline{2-6} & - & 4.4756 & 3.0705 & 4.2298 &  4.3603\\
\cline{2-6} &  2.0316e-10 & 1.2875e-11 &  &  &  \\
 \cline{2-6} & 5.0561 & 3.9800 &  &  & \\
\hline \multirow{3}{*}{$d=1.25e-01$}  &  9.9812e-04 & 4.3941e-05 & 5.2111e-06 & 2.7766e-07 & 1.3518e-08  \\
 \cline{2-6} & - & 4.5056 & 3.0759 & 4.2302 &  4.3604\\
\cline{2-6} &  4.0633e-10 & 2.5750e-11 &  &  &  \\
 \cline{2-6} & 5.0561 & 3.9800 &  &  & \\
\hline \multirow{3}{*}{$d=6.25e-02$}  &  4.7770e-01 & 8.8657e-05 & 1.0429e-05 & 5.5537e-07 & 2.7036e-08  \\
 \cline{2-6} & - & 12.3956 & 3.0877 & 4.2310 &  4.3605\\
\cline{2-6} &  8.1265e-10 & 5.1503e-11 &  &  &  \\
 \cline{2-6} & 5.0561 & 3.9800 &  &  & \\
\hline \multirow{3}{*}{$d=3.125e-02$}  &  4.7770e-01 & 1.8091e-04 & 2.0878e-05 & 1.1109e-06 & 5.4073e-08  \\
 \cline{2-6} & - & 11.3666 &3.1153 & 4.2322 &  4.3607\\
\cline{2-6} &  1.6253e-09 & 1.0301e-10 &  &  &  \\
 \cline{2-6} & 5.0561 & 3.9798 &  &  & \\
\hline \multirow{3}{*}{$d=1.5625e-02$}  &  4.7770e-01 & 3.6108e-01 & 3.7317e-01 & 2.2222e-06 & 1.0815e-07  \\
 \cline{2-6} & - & 0.4037 & -0.0475 & 17.3575 &  4.3609\\
\cline{2-6} &  3.2507e-09 & 2.0601e-10 &  &  &  \\
 \cline{2-6} & 5.0562 & 3.9799 &  &  & \\
\hline
\end{tabular}
\end{center}
\end{table}

\begin{table}[!ht]
\begin{center}
\begin{tabular}{|c|c|c|c|c|c|}

\hline \multirow{3}{*}{$d=7.1825e-03$}  &  4.7770e-01 & 3.6108e-01 & 3.7317e-01 & 4.0099e-01 & 2.1631e-07  \\
 \cline{2-6} & - & 0.4038 & -0.0475 & -0.1037 &  20.8220\\
\cline{2-6} &  6.5013e-09 & 4.1200e-10 &  &  &  \\
 \cline{2-6} & 5.0562 & 3.9800 &  &  & \\
\hline

\end{tabular}
\caption{\small\textit{Approximation errors $e_k$ in the infinity norm,  and approximation orders $p_k$  obtained at scale $k, k= 0,1,2,3,4,5,6$ for the detection of the corner discontinuity in the functions
 $f_d(x),$ with $d \ \in \{4,2,1,\frac{1}{2},\frac{1}{4},\frac{1}{8},\frac{1}{16},\frac{1}{32},\frac{1}{64},\frac{1}{128}\}.$
}} \label{tab2}
\end{center}
\end{table}

For the second experiment we measure the approximation errors and
the approximation orders when interpolating with the presented technique over the same nested grids as in the first experiment. This time we call $E_k$ to the approximation errors
$E_k:=||f_d(.)-I(.,f)||_{\infty}$ in infinity norm, and $P_k$ to the numerical approximation orders computed by $P_k=\frac{log_2{E_{k-1}}}{log_2{E_{k}}}.$ We carry out the computations for the different values of
the parameter $d,$ and the results can be seen in Table \ref{tab3}. The observations are quite coincident with those of the first experiment. The order of approximation becomes fourth order for values of $h_{max}< h_c$ and slightly larger values. When the parameter $d$ is smaller $h_c$ is also smaller and finer refinements must be used to capture the discontinuity and approximate adequately with fourth order.

\begin{table}[!ht]
\begin{center}
\begin{tabular}{|c|c|c|c|c|c|}
\hline
Cases  & \multicolumn{5}{|c|}{Refinement grid level $k$}  \\
\hline \multirow{4}{*}{Parameter $d$}  &  $E_0$ & $E_1$ & $E_2$ & $E_3$ & $E_4$ \\
 \cline{2-6} & $P_0$ & $P_1$ & $P_2$ & $P_3$ & $P_4$ \\
 \cline{2-6} & $E_5$ & $E_6$ &  & &  \\
 \cline{2-6} & $P_5$ & $P_6$ &  & &  \\
\hline \multirow{3}{*}{$d=4$}  &  1.5598e-04 & 3.5446e-05 & 1.8539e-06 & 8.5901e-08 & 4.8809e-09  \\
 \cline{2-6} & - & 2.1377 & 4.2570 & 4.4317 &  4.1375 \\
\cline{2-6} &  2.7073e-10 & 1.3498e-11 &  &  &  \\
 \cline{2-6} & 4.1722 & 4.3261 &  &  & \\
\hline \multirow{3}{*}{$d=2$}  &  1.5598e-04 & 3.5446e-05 & 1.8539e-06 & 8.5901e-08 & 4.8809e-09  \\
 \cline{2-6} & - & 2.1377 & 4.2570 & 4.4317 &  4.1374 \\
\cline{2-6} &  2.7073e-10 & 13498e-11 &  &  &  \\
 \cline{2-6} & 4.1722 & 4.3261 &  &  & \\
\hline \multirow{3}{*}{$d=1$}  &  1.5598e-04 & 3.5445e-05 & 18539e-06 & 8.5901e-08 & 4.8808e-09  \\
 \cline{2-6} & - & 2.1377 & 4.2570 & 4.4317 &  4.1375 \\
\cline{2-6} &  2.7073e-10 & 1.3498e-11 &  &  &  \\
 \cline{2-6} &  4.1722 & 4.3261 &  &  & \\
\hline \multirow{3}{*}{$d=5e-01$}  &  1.5598e-04 & 3.5446e-05 & 1.8539e-06 & 8.5901e-08 & 4.8809e-09  \\
 \cline{2-6} & - & 2.1377 & 4.2570 & 4.4317 &  4.1375\\
\cline{2-6} &  2.7073e-10 & 1.3498e-11 &  &  &  \\
 \cline{2-6} & 4.1722 & 4.3260 &  &  & \\
\hline \multirow{3}{*}{$d=2.5e-01$}  &  1.5598e-04 & 9.7662e-05 & 1.8539e-06 & 8.5901e-08 & 4.8809e-09  \\
 \cline{2-6} & - & 0.6755 & 5.7192 & 4.4317 &  4.1375\\
\cline{2-6} &  2.7073e-10 & 1.3498e-11 &  &  &  \\
 \cline{2-6} & 4.1722 &  4.3261 &  &  & \\
\hline \multirow{3}{*}{$d=1.25e-01$}  &  1.5598e-04 & 9.7662e-05 & 1.8539e-06 & 8.5901e-08 & 4.8809e-09  \\
 \cline{2-6} & - & 0.6755 & 5.7192 & 4.4317 &  4.1375\\
\cline{2-6} &  2.7073e-10 & 1.3498e-11 &  &  &  \\
 \cline{2-6} & 4.1722 & 4.3260 &  &  & \\

\hline

\end{tabular}
\end{center}
\end{table}

\begin{table}[!ht]
\begin{center}
\begin{tabular}{|c|c|c|c|c|c|}
\hline \multirow{3}{*}{$d=6.25e-02$}  &  9.7337e-04 & 9.7662e-05 & 1.8539e-06 & 8.5901e-08 & 4.8809e-09  \\
 \cline{2-6} & - & 3.3171 & 5.7192 & 4.4317 &  4.1375\\
\cline{2-6} &  2.7073e-10 & 1.3498e-11 &  &  &  \\
 \cline{2-6} & 4.1722 & 4.3261 &  &  & \\
\hline \multirow{3}{*}{$d=3.125e-02$}  &  6.0132e-04 & 9.7662e-05 & 1.8538e-06 & 8.5901e-08 & 4.8809e-09  \\
 \cline{2-6} & - & 2.6223 &5.7192 & 4.4317 &  4.1375\\
\cline{2-6} &  2.7073e-10 & 1.3498e-11 &  &  &  \\
 \cline{2-6} & 4.1722 & 4.3261 &  &  & \\
\hline
\hline \multirow{3}{*}{$d=1.5625e-02$}  &  8.1175e-04 & 9.3378e-05 & 1.1919e-05 & 8.5901e-08 & 4.8809e-09  \\
 \cline{2-6} & - & 3.1199 & 2.9698 & 7.1163 &  4.1375\\
\cline{2-6} &  2.7073e-10 & 1.3498e-11 &  &  &  \\
 \cline{2-6} & 4.1722 & 4.3261 &  &  & \\
\hline
\hline \multirow{3}{*}{$d=7.1825e-03$}  &  4.7140e-04 & 1.1970e-04 & 6.2468e-06 & 5.6274e-06 & 4.8809e-09  \\
 \cline{2-6} & - & 1.9776 & 4.2601 & 0.1506 &  10.1711\\
\cline{2-6} &  2.7073e-10 & 1.3498e-11 &  &  &  \\
 \cline{2-6} & 4.1722 & 4.3260 &  &  & \\
\hline

\end{tabular}
\caption{\small\textit{Approximation errors $E_k$ in the infinity norm, and approximation orders $P_k$  obtained at scale $k, k= 0,1,2,3,4,5,6$ for the interpolation with the presented technique adapted to corner discontinuities of the functions
 $f_d(x),$ with $d \ \in \{4,2,1,\frac{1}{2},\frac{1}{4},\frac{1}{8},\frac{1}{16},\frac{1}{32},\frac{1}{64},\frac{1}{128}\}.$
}} \label{tab3}
\end{center}
\end{table}

Notice that $\sigma=2$ is not included in the range of values satisfying the hypothesis of Theorem \ref{teo_approF}, proving that the conditions on the grid are sufficient to prove the result,
but not necessary.

\section{Conclusions} \label{sec7}

In this article we have extended a corner detection mechanism \cite{ACDD} to $\sigma$ quasi-uniform grids, and in turn also the corresponding SR-type reconstruction algorithm.
The theoretical approximation orders of these algorithms have been studied, proving that they remain the same as with uniform grids for a wide range of $3-$local $\sigma$ quasi-uniform grids.
Numerical experiments have been carried out, examining the validity and extension of the presented results.

\end{document}